\documentclass[reqno,12pt]{amsart}
\usepackage{euscript}
\usepackage{yhmath} 
\DeclareMathAccent{\widehat}{\mathord}{largesymbols}{"62}
\DeclareFontFamily{U}{mathx}{\hyphenchar\font45}
\DeclareFontShape{U}{mathx}{m}{n}{
      <5> <6> <7> <8> <9> <10>
      <10.95> <12> <14.4> <17.28> <20.74> <24.88>
      mathx10
      }{}
\DeclareSymbolFont{mathx}{U}{mathx}{m}{n}
\DeclareFontSubstitution{U}{mathx}{m}{n}
\DeclareMathAccent{\widecheck}{0}{mathx}{"71}

\def\ve{\mathsf{e}}
\def\vu{\mathsf{u}}
\def\vv{\mathsf{v}}
\def\vz{\mathsf{z}}
\def\vw{\mathsf{w}}
\def\vh{\mathsf{h}}
\def\vn{\mathsf{n}}

\def\bnu{\mathbf{\nu}}

\def\rM{\mathsf{M}}
\def\rG{\mathsf{G}}
\def\rS{\mathsf{S}}
\def\ce{\underbar{e}}
\def\trim{\widecheck} 
\def\Et#1{\trim{E_{#1}}}
\def\Etn{\Et{n}}
\def\et#1{\trim{e_{#1}}}
\def\Ht#1{\trim{H_{#1}}}
\def\trimb#1{\trim{\overbar{#1}}}
\def\opoint{{\trim{o}}}
\def\clip{\widetilde} 
\def\clipb#1{\clip{\overbar{#1}}}
\def\split{\wideparen}
\def\lgp{\left\lgroup}
\def\rgp{\right\rgroup}

\def\Phihat{\widehat{\Phi}}
\def\A{\mathfrak{A}}
\def\Ahat{\widehat{\A}}

\def\a{\alpha}
\def\b{\beta}
\def\w{\omega}
\def\sech{\operatorname{sech}}
\def\csh{\cosh t\,}
\def\ssh{\sinh t\,}
\def\di{\partial}
\def\dibar{\overbar{\partial}}
\def\C{\mathbb{C}}
\def\CP{\mathbb{C}P}
\def\RP{\mathbb{R}P}
\def\R{\mathbb{R}}
\def\M{\mathcal M}
\def\B{\mathcal B}
\def\F{\mathcal F}
\def\scrA{\EuScript A}\def\scrB{\EuScript B}

\def\calD{\mathcal D}
\def\calE{\mathcal E}
\def\calH{\mathcal H}

\def\calN{\mathcal N}
\def\calX{\EuScript X}
\def\scrN{\EuScript N}
\def\Mhat{\widehat{M}}
\def\rJ{\mathsf{J}}

\def\scrN{\EuScript{N}}

\newcommand{\overbar}[1]{\mkern 0.8mu\overline{\mkern-0.8mu#1\mkern-0.8mu}\mkern 0.8mu}

\def\Zbar{\overbar{Z}}
\def\Wbar{\overbar{W}}

\def\Kahler{\Omega}
\def\Hvol{\Theta} 

\def\ri{\mathrm{i}}
\renewcommand\Im{\operatorname{Im}}
\renewcommand\Re{\operatorname{Re}}
\def\rk{\operatorname{rk}}
\def\II{\operatorname{II}}
\def\intprod{\mathbin{\raisebox{.4ex}{\hbox{\vrule height .5pt width
5pt depth 0pt %
         \vrule height 3pt width .5pt depth 0pt}}}}

\newtheorem{thm}{Theorem}
\newtheorem{lemma}[thm]{Lemma}
\newtheorem{prop}[thm]{Proposition}
\newtheorem{cor}[thm]{Corollary}
\theoremstyle{definition}

\newtheorem{egg}{Example}

\begin{document}
\title{Austere Submanifolds in $\CP^n$}
\author{Marianty Ionel}
\address{Institute of Mathematics, Federal University of Rio de Janeiro, Brazil}
\author{Thomas A. Ivey}
\address{Dept. of Mathematics, College of Charleston, 66 George St., Charleston SC 29424}

\begin{abstract}
For an arbitrary submanifold $M \subset \CP^n$ we determine conditions under which it is {\em austere},
i.e., the normal bundle of $M$ is special Lagrangian with respect to Stenzel's Ricci-flat K\"ahler metric on $T\CP^n$.
We also classify austere surfaces in $\CP^n$.
\end{abstract}

\date{July 15, 2015}
\maketitle

\section{Introduction}
Special Lagrangian submanifolds were introduced in 1982 by Harvey and Lawson in their seminal paper \cite{HL}.
They studied them in the more general context of calibrated submanifolds, which are a special class of minimal submanifolds. Calibrated submanifolds, in particular special Lagrangian submanifolds, play an important role in mirror symmetry and  they have lately been the object of extensive study.  Most of the earlier research has focused on special Lagrangian submanifolds in $\C^n$ (see Joyce \cite{JHol} and the extensive references contained therein).

Let $\calX$ be a Calabi-Yau manifold of complex dimension $n$ with K\"ahler form $\Kahler$ and holomorphic volume form $\Hvol$.  Recall that an oriented submanifold $L$ of real dimension $n$ is {\em special Lagrangian} if it is calibrated by $\Re \Hvol$.  Harvey and Lawson showed that $L$ is special Lagrangian if and only if $\Kahler|_L\equiv 0$ and $\Im \Hvol |_L\equiv 0$.  (The same is true if we replace $\Hvol$ by $e^{\ri \phi} \Hvol$, in which case $L$ is said
to be special Lagrangian {\em with phase} $e^{\ri \phi}$.)  In the same paper \cite{HL}, Harvey and Lawson exhibited a construction of special Lagrangian submanifold using bundles. Specifically, they showed that the conormal bundle $N^*M$ of an immersed submanifold $M^k\subset \R^n$ is special Lagrangian in $\C^n\cong T^*\R^n$ (with phase depending on $k$ and $n$) if and only if $M^k$ is {\em austere} in $\R^n$, i.e., the second fundamental form of $M$ in any normal direction has its eigenvalues symmetrically arranged around 0.  (This is equivalent to saying that all the odd-degree symmetric polynomials in the eigenvalues of the second fundamental form vanish identically.) Note that the austere condition implies that $M$ is minimal, but in general is stronger than minimality.

In the early 1990s, Stenzel \cite{Sthesis,Stpaper} showed that the cotangent bundle of any compact rank one symmetric space can be endowed with a Ricci-flat metric, which is now called the {\em Stenzel metric}. Particular cases of Stenzel metrics were initially discovered by Eguchi-Hanson on the cotangent bundle $T^* S^2$ and by Candela-de la Ossa on $T^*S^3$. In \cite{KM} Karigiannis and Min-Oo generalized Harvey and Lawson's construction to the cotangent bundle of $S^n$ carrying the Stenzel Ricci-flat metric.  Specifically, they showed that the conormal bundle over an immersed submanifold $M\subset S^n$ is special Lagrangian with respect to some phase if and only if  all the odd-degree symmetric polynomial in the eigenvalues of the second fundamental form, in any normal direction, vanish identically.  In other words, this is the same austere condition as Harvey and Lawson found in \cite{HL} for $\R^n$.  This is  perhaps surprising, since the complex structure on $T^*S^n$ is not the standard one (as it is in the case of $\C^n\cong T^*\R^n$) but instead is obtained by identifying it with a complex affine hyperquadric in $\C^{n+1}$.

In this paper, we further generalize the Harvey and Lawson construction to the case of $T^*\CP^n$, the cotangent bundle of complex projective space.
We define $M\subset \CP^n$ to be {\it austere} if its conormal bundle $N^*M$ is special Lagrangian in $T^*\C P^n$, with respect to the Stenzel metric.
We will calculate conditions on the second fundamental form of $M$ that are necessary and sufficient for $M$ to be austere.
(In fact, we will work on the normal bundle $NM$, using the standard metric on $\CP^n$ to
 identify $T^* \CP^n$ with $T \CP^n$.)

We now give a brief description of the contents of the paper.
\begin{itemize}
\item
In section \ref{lagsect}, we define a mapping $\Phi$ that
 identifies $T\CP^n$ with a Stein manifold which is the complement of a complex quadric in $\C P^n\times \C P^n$.
 (We do this in order to utilize the convenient expression of the Stenzel K\"ahler form given by Lee \cite{Lee} in terms of coordinates on $\C P^n\times \C P^n$.)  We calculate the differential of the restriction of $\Phi$ to $N M$ using moving frames.
We prove that if $M\subset \C P^n$ is an arbitrary immersed submanifold, then $N M$ is a Lagrangian submanifold of $T\CP^n$
with respect to the Stenzel K\"ahler form (see Prop. \ref{lagness}).\footnote{
Note that while it is a basic result that the conormal bundle of any submanifold of a space $\calX$ is a Lagrangian submanifold of $T^*\calX$ with
respect to its standard symplectic form, our Prop. \ref{lagness} is non-trivial since the Stenzel K\"ahler form is not the standard symplectic structure.}
\item
In section \ref{ausect}, we determine the conditions under which an immersed submanifold $M\subset \CP^n$ is austere (see Theorem \ref{austerethm}).
A corollary of this result is that if $M\subset \C P^n$ is an arbitrary complex submanifold, then $M$ is austere (see Corollary \ref{holcor}).
\item
Finally, in section \ref{surfsect}, we classify the austere surfaces in $\CP^n$, showing that they must be either holomorphic curves or totally geodesic (see Theorem \ref{surfthm}).
\end{itemize}
While it might seem that there is a dearth of non-holomorphic examples of austere submanifolds in $\CP^n$, in subsequent papers we will
investigate the solution space of the austere condition for real hypersurfaces in this geometry, and exhibit
new examples of austere hypersurfaces in $\CP^2$ and higher dimensions.

Before proceeding with the calculations leading up to Proposition \ref{lagness}, we need to make a few remarks:
\begin{enumerate}
\item
As indicated above, our calculations will be made using moving frames.  Because the members of the moving frame are easier to differentiate
when they take value in a fixed vector space, we will do most calculations
on $\Mhat \subset S^{2n+1} \subset \C^{n+1}$, which is the inverse image of $M$ under the Hopf fibration
$\pi: S^{2n+1} \to \CP^n$.
(This is the restriction of the projectivization map from $\C^{n+1}/\{0\}$ to $\CP^n$, which we also denote by $\pi$.)
Similarly, we will do calculations on $NM \subset T \CP^n$ by regarding $T \CP^n$ as a quotient space
and working on the inverse image.
In more detail, let
$$B =\{ (\zeta,\xi) \in \C^{n+1} \times \C^{n+1}\, | \, \zeta \ne 0, \xi \cdot \overbar{\zeta}=0\},$$
wherein the dot product is $\C$-bilinear.  Recall that for a one-dimensional subspace $L \subset \C^{n+1}$, $T_L \CP^n$ is canonically defined
as the set of $\C$-linear maps $f$ from $L$ to the quotient
vector space $\C^{n+1}/L$.  If $f$ is such a map and $\zeta$ is a nonzero point on $L$, there is a unique
$\xi$ which projects to $f(\zeta)$ in the quotient vector space and satisfies $\xi \cdot \overbar{\zeta}=0$.
Thus,
 $T\CP^n \cong B/\C^*$, where the $\C^*$ action on $B$ is
$(\zeta,\xi)\mapsto (\lambda\zeta, \lambda \xi)$ for $\lambda\in \C^*$.

\medskip
\item Although the cotangent bundle of any complex manifold $\calX$
has a standard complex structure (obtained by identifying it with the bundle of $(1,0)$-forms),
the complex structure underlying the Stenzel metric on $T^*\CP^n$ is not the standard one.
For example, under the mapping $\Phi$ the image of the zero section is a totally real submanifold.

\medskip
\item For an arbitrary submanifold $M \subset \CP^n$ and $p\in M$, let $\calH_p$ and $\calN_p$
be maximal $\rJ$-invariant subspaces of $T_p M$ and $N_p M$ respectively (where $\rJ$ denotes the ambient complex structure).
We'll assume that $\calH = \underset{p \in M}\bigcup \calH_p$ and $\calN = \underset{p \in M}\bigcup \calN_p$
are smooth sub-bundles of $TM$ and $NM$, and let $\calD$ and $\calE$ be their respective orthogonal complements.
Thus, we have an orthogonal splitting
\begin{equation}\label{Tsplit}
T\CP^n\vert_M = TM \oplus NM = (\calH \oplus \calD) \oplus (\calE \oplus \calN)
\end{equation}
such that $\calD \oplus \calE$ is $\rJ$-invariant and $\rk \calD = \rk \calE \le n$.
As we will see in the statement of Theorem 4, the austere condition along a fiber of $NM$ depends on how the corresponding
normal vector splits into components in $\calE$ and $\calN$.
\end{enumerate}

We gratefully acknowledge the comments and helpful discussions with the following mathematicians
during our work on this paper:
Henri Anciaux,
Ronan Conlon,
Spiro Karigiannis,
JaeHyouk Lee,
Conan Leung,
Paul-Andi Nagy,
and  Pat Ryan.

\section{Lagrangian Submanifolds via Normal Bundles}\label{lagsect}
Throughout what follows, $M \subset \CP^n$ will denote a submanifold of real dimension $k$.
As above, let $\Mhat = \pi^{-1}(M)$ where $\pi:S^{2n+1}\to \CP^n$ is the Hopf fibration.
(We will regard points of $S^{2n+1}$, as well as tangent vectors to the sphere, as
vectors in $\C^{n+1}$.)
Let $\F$ denote the bundle of oriented orthonormal frames $(e_0, \ldots, e_{2n})$ along $\Mhat$ that are {\em adapted} in the sense
that $e_0, \ldots, e_k$ are tangent to $\Mhat$ and $e_0 = \ri \vz$ is tangent to the fiber of $\pi$ at $\vz \in \Mhat$.
(It follows that $e_1, \ldots, e_{2n}$ are orthogonal to the fiber of $\pi$, and so are called {\em horizontal} vectors.)
The fibers of $\pi:\Mhat \to M$ are, of course, orbits of the $S^1$ action $\vz \mapsto e^{\ri \theta} \vz$.  This
action extends to $\F$,  by simultaneously multiplying all frame vectors by $e^{\ri \theta}$, and preserves
horizontality.


Given any $\vz \in \Mhat$ and any normal vector $\vn \in N_\vz \Mhat$ (necessarily horizontal), there exists an adapted frame at $\vz$ such that
$\vn = t e_{2n}$ for some $t\in \R$.
Thus, the mapping
$$\varrho:((\vz, e_0, e_1, \ldots, e_{2n}),t) \mapsto t e_{2n} \in T_\vz S^{2n+1}$$
is a submersion from $\F \times \R$ to $N\Mhat$.  Because any normal vector $\bnu \in N_{\pi(\vz)} M$ has a horizontal
lift in $N_\vz \Mhat$, there is also a submersion $\Pi:N\Mhat \to NM$ defined by
$\vn \mapsto \pi_* \vn$.  We define yet another submersion
$$\rho:((\vz, e_0, e_1, \ldots, e_{2n}),t) \mapsto (\vz, t e_{2n}) \in B$$
that lifts the natural inclusion $\imath:NM\to T \CP^n$.  In other words, we
have a commutative diagram
\begin{center}
\setlength{\unitlength}{4pt}
\begin{picture}(35,25)(10,0)
\put(16,20){\makebox(0,0){$\F\times \R$}}
\put(20,20){\vector(1,0){8}}
\put(23,22){\makebox(0,0){$\rho$}}
\put(30,20){\makebox(0,0){$B$}}
\put(30,17){\vector(0,-1){13}}
\put(16,19){\vector(0,-1){5}}
\put(13,16){$\varrho$}
\put(14,10){$N\Mhat$}
\put(14,0){$NM$}
\put(16,9){\vector(0,-1){6}}
\put(13,5){$\Pi$}
\put(24,3){\makebox(0,0){$\imath$}}
\put(20,1){\vector(1,0){7}}
\put(28,0){$T\CP^n$}
\end{picture}
\end{center}
where the vertical map at right is quotient by the $\C^*$ action.

Let $\M$ denote the open subset in $\CP^n \times \CP^n$ defined in homogeneous
coordinates by $\displaystyle\sum_{i=0}^n z_i w_i \ne 0$.  (This space is denoted by $M^{2n}_{II}$ in \cite{Lee}.)
 To generate an embedding of $NM$ as a submanifold in the Stein manifold $\M$, we will
compose $\rho$ with a map $\Phihat : B \to \C^{n+1}\times \C^{n+1}$ defined by
\begin{equation}\label{defofPhi}
\Phihat(\zeta,\xi) = \left( (\cosh \mu) \zeta + \ri \left( \dfrac{\sinh \mu}{\mu} \right) \xi;
(\cosh \mu) \overbar\zeta + \ri \left( \dfrac{\sinh \mu}{\mu} \right) \overbar\xi\right),
\quad \mu = \dfrac{|\xi|}{|\zeta|}.
\end{equation}
(This is adapted from the work of Sz\"oke \cite{Sz}.)  We will write elements of $\C^{n+1} \times \C^{n+1}$ as
an ordered pair of row vectors in $\C^{n+1}$ separated by a semicolon.

It is easy to check that, relative to the $\C^*$ action on $B$ and projectivization
on each factor in $\C^{n+1}\times \C^{n+1}$, $\Phihat$ covers a well-defined
embedding $\Phi: T\CP^n \to \CP^n \times\CP^n$ that identifies $T\CP^n$ bijectively with $\M$.
We will compute the differential of the composition of $\Phi$ with the inclusion $\imath$ by computing
the differential of the composition of maps along the top edge of the following diagram
\begin{equation}\label{bigcommute}
\setlength{\unitlength}{4pt}
\begin{picture}(65,25)(5,0)
\put(16,20){\makebox(0,0){$\F\times \R$}}
\put(20,20){\vector(1,0){8}}
\put(23,22){\makebox(0,0){$\rho$}}
\put(30,20){\makebox(0,0){$B$}}
\put(30,17){\vector(0,-1){13}}
\put(32,20){\vector(1,0){16}}
\put(41,22){\makebox(0,0){$\Phihat$}}
\put(50,19){$\C^{n+1}\times \C^{n+1}$}
\put(57,17){\vector(0,-1){13}}
\put(60,17){\vector(1,-1){13}}
\put(16,19){\vector(0,-1){5}}
\put(13,16){$\varrho$}
\put(14,10){$N\Mhat$}
\put(14,0){$NM$}
\put(16,9){\vector(0,-1){6}}
\put(13,5){$\Pi$}
\put(24,3){\makebox(0,0){$\imath$}}
\put(20,1){\vector(1,0){7}}
\put(28,0){$T\CP^n$}
\put(16,9){\vector(0,-1){6}}
\put(37,1){\vector(1,0){11}}
\put(42,3){\makebox(0,0){$\Phi$}}
\put(50,0){$\CP^n \times \CP^n$}
\put(66,1){\vector(1,0){7}}
\put(69,3){\makebox(0,0){$\A$}}
\put(69,12){\makebox(0,0){$\Ahat$}}
\put(75,0){$\C^{2n}$}
\end{picture}
\end{equation}
in which $\A$ denotes an affine coordinate chart (to be specified below) and $\Ahat$ is the corresponding map in terms of homogeneous coordinates.
As we will see at the end of the next section, the differential along the top edge annihilates vectors that are tangent
to the fibers of the map $\Pi \circ \varrho$.

\subsection{Geometry of the Frame Bundle}
On $\F$ we define a set of real-valued 1-forms
by expanding the derivatives of the basepoint function and frame vectors in terms of the basis (over $\R$) for $\C^{n+1}$ provided
by the frame itself.  Taking the index ranges $0\le a,b\le k$ and $k+1 \le \kappa, \lambda \le 2n$, and using the summation convention, we write
\begin{equation}\label{dees0}
d\vz = e_a \w^a, \qquad d e_a = -\vz \w^a + e_b \psi^b_a + e_\lambda \psi^\lambda_a, \qquad
d e_\kappa = e_b \psi^b_\kappa + e_\lambda \psi^\lambda_\kappa.
\end{equation}
(In these equations each term is a vector-valued 1-form, i.e., a section of $\C^{n+1} \otimes T^* \F$.  While we omit the tensor
product symbol in these equations and write the vector factor on the left, it will occasionally be convenient to write the
vector on the right, as in \eqref{bigd}, \eqref{daff}, \eqref{daffy} below.)
The last equation in \eqref{dees0} has no $\vz$ terms on the right-hand side because
$$0 = \langle  d\vz, e_\kappa \rangle = - \langle \vz, d e_\kappa\rangle,$$
where $\langle, \rangle$ denotes the real-valued Euclidean inner product on $\C^{n+1}$.

The full oriented orthonormal frame bundle of $S^{2n+1}$ may be identified with the special orthogonal group $SO(2n+2)$,
by taking the basepoint and the frame vectors as successive rows of an orthogonal matrix.
However, since by default we will regard the frame vectors as taking value in $\C^{n+1}$, to make the identification precise we introduce the
convention that for a vector $\vv \in \C^N$,
$$\split{\vv} := (\Re\vv; \Im\vv) \in \R^{2N}.$$
Then, in terms of this notation,
\begin{equation}\label{sonny}
\begin{pmatrix}\split{\vz} \\ \split{e_0} \\ \split{e_1} \\ \vdots \\ \split{e_{2n}}\end{pmatrix} \in SO(2n+2).
\end{equation}
In this way, $\F$ is identified with a submanifold in $SO(2n+2)$, and in fact the 1-forms defined by \eqref{dees0} are the components
of the Maurer-Cartan form of $SO(2n+2)$, pulled back to $\F$.

Not all of these 1-forms are linearly independent on $\F$.  We will need to clarify
the dependencies among the 1-forms on $\F$ involved in the differential of the components $\vz$ and $e_{2n}$ of the map $\rho$, which are
$$\w^0, \w^\a, \psi^0_{2n}, \psi^\a_{2n}, \psi^\nu_{2n}.
$$
(Here, we introduce new index ranges $1 \le \alpha, \beta \le k$ and $k+1\le   \mu, \nu < 2n$.)
Recalling that $e_0 = \ri\vz$, we have
$$\psi^0_{2n} = \langle e_0, de_{2n}\rangle = -\langle \vz, d(\ri e_{2n})\rangle = \langle d\vz, \ri e_{2n}\rangle.$$
Thus, if we define functions $r_\a = \langle \ri e_{2n}, e_\a\rangle$ on $\F$, then $\psi^0_{2n} = r_\a \w^\a$.
(Note that if $M$ is a complex submanifold then all $r_\a=0$.)
We also have
$$0 = d\langle e_{2n}, d \vz \rangle = \psi^a_{2n} \wedge \w^a = (\psi^\a_{2n} - r_\a \w^0) \wedge \w^\a,$$
which shows that
$$\psi^\a_{2n} = r_\a \w^0-h_{\a\b}\w^\b$$
for some functions $h_{\a\b} = h_{\b\a}$.  In fact, one can check that if we define
$\ce_{\a} = \pi_* e_\a$ and $\ce_{2n}=\pi_* e_{2n}$, then
 $h_{\a\b}= \ce_{2n} \cdot \II(\ce_\a, \ce_\b),$
where $\II$ is the second fundamental form of $M$ at the point $\pi(\vz)$.
Putting these results together, we have
\begin{equation}\label{dees}
d\vz = \ri \vz \w^0 + e_\a \w^\a, \qquad d e_{2n} = \ri  \vz (r_\a \w^\a) + e_\a (r_\a \w^0 -h_{\a\b} \w^\b) + e_\mu \psi^\mu_{2n}.
\end{equation}

\subsection{Computing the Differential}
We now use the above formulas to compute
\begin{multline*}
d(\Phihat \circ \rho) = (\csh d\vz + \ri \ssh de_{2n}+(\ssh \vz + \ri \csh e_{2n}) dt;\\
\csh d\overbar{\vz} + \ri \ssh d\overbar{e_{2n}} + (\ssh \overbar{\vz} +\ri \csh\overbar{e_{2n}}) dt),
\end{multline*}
giving
\begin{multline}\label{bigd}
d(\Phihat \circ \rho) =
\csh [\w^0,\, \w^\a,\, \psi^\mu_{2n},\, dt] \\
\otimes \lgp \lgp \begin{array}{cccc} \ri & \ri \tau r_\b & 0 & 0 \\
            -\tau r_\a & \delta_{\a\b} -\ri \tau h_{\a\b} & 0 & 0\\
            0 & 0 & \ri\tau \delta_{\mu\nu} & 0 \\
            \tau & 0 & 0 & \ri\end{array}\rgp
\begin{bmatrix} \vz \\ e_\b \\ e_\nu \\ e_{2n}\end{bmatrix};
\lgp \begin{array}{cccc}
             -\ri & \ri \tau r_\b & 0 & 0 \\
            \tau r_\a & \delta_{\a\b}-\ri\tau h_{\a\b} & 0 & 0\\
            0 & 0 & \ri\tau \delta_{\mu\nu} & 0 \\
            \tau & 0 & 0 & \ri \end{array}\rgp
\begin{bmatrix} \overbar{\vz}\\ \overbar{e_\b} \\ \overbar{e_\nu} \\ \overbar{e_{2n}}\end{bmatrix} \rgp.
\end{multline}
(The matrices are partitioned into block form, with column and row widths $1$, $k$, $n-k-1$ and $1$; we have also introduced
the abbreviation $\tau = \tanh t$.)

We will use the isometry group $U(n+1)$ to simplify these matrices, by moving any horizontal vector in $N\Mhat$ into
a standard position.  For, given any point $\vz \in \Mhat$ and a horizontal vector $\vh \in T_{\vz} S^{2n+1}$, we
can arrange that
$$\vz = \begin{bmatrix} 1 & 0 & \ldots & 0\end{bmatrix}, \qquad \vh =\begin{bmatrix} 0 & \ldots & 0 & \ri\end{bmatrix}.$$
Thus, from now on we will assume that $\vz = E_0$ and $e_{2n} = \ri E_n$,
where $E_0, \ldots, E_n$ denote the elementary basis row vectors of $\C^{n+1}$.  Then
$$\Phihat\circ \rho((\vz, e_0,\ldots,e_{2n}),t) = (\csh, 0, \ldots, 0, -\ssh; \csh, 0, \ldots, 0, \ssh ).$$

Let $o \in \C^{n+1} \times \C^{n+1}$ be the point on the right-hand side.  For any $t$ this is in the domain of the map
$$\Ahat: (z_0, \ldots, z_n; w_0, \ldots, w_n) \mapsto \left( \dfrac{z_1}{z_0}, \ldots, \dfrac{z_n}{z_0};
                                                    \dfrac{w_1}{w_0}, \ldots, \dfrac{w_n}{w_0} \right).$$
which covers the affine coordinate chart $\A$.
We compute the differential of $\Ahat$, and evaluate at $o$:
$$d\Ahat_o = \sech t(dz_1, \ldots, dz_{n-1}, dz_n+\tau dz_0; dw_1, \ldots, dw_{n-1}, dw_n -\tau dw_0).$$
To compute the differential of $\Ahat\circ \Phihat\circ\rho$, we apply the $\C$-linear map $d\Ahat_o$ to the
vector factors in the tensor product in \eqref{bigd}.  For example,
$$d\Ahat_o(e_\a;0) =(\et\a; 0),\qquad d\Ahat_o(e_{2n};0) = (\ri \Etn;0), \qquad d\Ahat_o(\vz;0) = (\tau\Etn;0),
\qquad d\Ahat_o(0;\overbar{\vz})= (0,-\tau\Etn),
$$
where the $\,\trim{}\,$ indicates the result of deleting the first entry from a row vector in $\C^{n+1}$,
and we use the fact that, because the vectors $e_\a$ and $e_\nu$  are horizontal at $\vz=E_0$, their
first entries are zero.  Computing in this way, we obtain
\begin{equation}\label{daff}
d(\Ahat\circ\Phihat\circ\rho) = [\w^0,\, \w^\a,\, \psi^\mu_{2n},\, dt] \otimes \\
\lgp\begin{array}{ccc}
\ri \tau (\Etn + r_\b \et\b) &;& \ri \tau (\Etn+r_\b \trim{\overbar{e_\b}}) \\
\et\a -\ri \tau h_{\a\b} \et\b -\tau^2 r_\a\Etn &;& \trim{\overbar{e_\a}} - \ri \tau h_{\a\b}\trim{\overbar{e_\b}}-\tau^2 r_\a \Etn \\
\ri \tau \et\mu &;& \ri\tau \trim{\overbar{e_\mu}} \\
(\tau^2-1)\Etn &;& (1-\tau^2)\Etn
\end{array}
\rgp.
\end{equation}
While the 1-forms $\w^0, \w^\a, \psi^\mu_{2n}$ and $dt$ are linearly independent and span the
semibasic 1-forms for $\varrho$, it is evident from the first equation in  \eqref{dees} that $\w^0$ is not semibasic for $\Pi\circ \varrho$.
In fact, if $\vv$ is the vector field on $\F\times \R$ that generates the $S^1$ action under which
$NM = N\Mhat/S^1$, then
$$\vv \intprod \w^0 = 1, \qquad \vv \intprod \w^\a =0, \qquad \vv\intprod \psi^\mu_{2n} = r_\mu:=\langle \ri e_{2n},e_\mu\rangle.$$
Using these formulas, it is easy to check that $\vv$ is in the kernel of the differential \eqref{daff}.
In fact, the $r_\mu$ give the coefficients under which the top row of the matrix is a linear combination
of the third set of rows below it, since $\Etn + r_\b \et\b  = -\ri \et{2n} + r_\b \et\b = -r_\mu e_\mu$.

Thus, in terms of the diagram \eqref{bigcommute}, $\Ahat \circ\Phihat\circ\rho$ covers a well-defined map $\A \circ \Phi\circ\imath$
from $NM$ to $\C^{2n}$.  Matters being so, we will expand the differential just in terms of the 1-forms
$\w^\a, dt$ and $\widetilde \psi^\mu_{2n}:=\psi^\mu_{2n}-r_\mu\w^0$, as
\begin{multline}\label{daffy}
d(\Ahat\circ\Phihat\circ\rho) = [\w^\a,\, \widetilde \psi^\mu_{2n},\, dt] \otimes
\lgp\begin{array}{ccc}
\et\a -\ri \tau h_{\a\b} \et\b -\tau^2 r_\a\Etn &;& \trim{\overbar{e_\a}} - \ri \tau h_{\a\b}\trim{\overbar{e_\b}}-\tau^2 r_\a \Etn \\
\ri \tau \et\mu &;& \ri\tau \trim{\overbar{e_\mu}} \\
(\tau^2-1)\Etn &;& (1-\tau^2)\Etn
\end{array}
\rgp.
\end{multline}

\subsection{The Stenzel K\"ahler Form}
A convenient explicit description of the Stenzel metric in local coordinates on $\CP^n \times \CP^n$ is given by
T-C. Lee \cite{Lee}; we briefly reproduce it here for the sake of the calculations in \S\ref{lagsec}.
Lee defines two functions on $\C^{n+1} \times \C^{n+1}$,
$$\scrA = \sum_{j,k=0}^n |z_j w_k|^2, \qquad \scrB = \vz \cdot \vw = \sum_{j=0}^n z_j w_j,$$
which are homogeneous of degree 4 and 2 respectively; then the exhaustion function $\scrN = \scrA/|\scrB|^2$
is well-defined and smooth on $\M$.
The K\"ahler potential $f(\scrN)$ for the Stenzel metric satisfies $f' = \scrN^{-1/2}$.
Using this, we calculate
the K\"ahler form in terms of affine coordinates $Z_a = z_a/z_0$ and $W_a = w_a/w_0$, where we now take $1\le a, b \le n$.

To start with,
$$\dibar \scrA = (1+|W|^2) Z_b\, d\Zbar_b + (1+|Z|^2) W_b\, d\Wbar_b$$
where $|Z|^2 = \sum_a Z_a \Zbar_a$ and $|W|^2 = \sum_a W_a \Wbar_a$; then
\begin{align}\di \dibar\scrA &= (1+|W|^2) dZ_b \wedge d\Zbar_b + \Wbar_a Z_b\, dW_a \wedge d\Zbar_b + \Zbar_a W_b\, dZ_a \wedge d\Wbar_b
+ (1+|Z|^2) dW_b \wedge d\Wbar_b \notag \\
&= (dZ \, dW) \wedge \begin{pmatrix} (1+|W|^2) I_n & \Zbar^T W \\ \Wbar^T Z & (1+|Z|^2) I_n \end{pmatrix} \begin{pmatrix} d\Zbar^T \\ d\Wbar^T\end{pmatrix},
\label{Ahermit}
\end{align}
where we regard $Z$ and $W$ as row vectors of length $n$, $I_n$ is the $n\times n$ identity matrix,
and ${}^T$ indicates transpose.
Following Lee, we identify the $(1,1)$-form $\di \dibar\scrA$ with the hermitian matrix in \eqref{Ahermit}
that gives its coefficients with respect to the affine coordinate differentials;
similarly, we identify $\di \scrA$ with the row vector $\begin{pmatrix} (1+|W|^2) \Zbar\,;\, (1+|Z|^2) \Wbar\end{pmatrix}$ of length $2n$.
With these conventions, we identify $\di\dibar f$ with the hermitian matrix
$$\rG = \dfrac{f'}{|\B|^2}\left[ \di \dibar \scrA - \dfrac1{\scrA} (\di \scrA)^T \dibar \scrA +
\left(\dfrac{f''}{|\B|^2 f' } + \dfrac1{\scrA}\right)\left( \di \scrA -(\scrA/\scrB)\di \scrB\right)^T \left(\dibar \scrA - (\scrA/\overbar{\scrB}) \dibar \overbar{\scrB}\right)\right].$$
We will only need the value of the metric at the point $\opoint=\Ahat(o)$, where
$W_n=\tau=\tanh t$, $Z_n = -\tau$ and all other coordinates are zero.  At this point, we have
\begin{equation}\label{Gpoint}
\rG = \dfrac1{1-\tau^4}\left[ (1+\tau^2)I_{2n}
+ \begin{pmatrix} (q-\tau^2) \rM & -q \rM \\ -q \rM & (q-\tau^2) \rM \end{pmatrix}\right],
\end{equation}
where
$$\rM= \Etn^T \Etn, \qquad \text{and} \quad q = \dfrac{2\tau^2}{(1-\tau^2)^2}.$$

\subsection{Checking Lagrangian-ness}\label{lagsec}
In this section, we prove:
\begin{prop}\label{lagness}  If $M \subset \CP^n$ is an arbitrary submanifold, then $\Phi(NM)$ is a Lagrangian
submanifold of $\M$.
\end{prop}
\begin{proof}
We will identify a real tangent vector $\vv \in T_\opoint \M$
with the row vector in $\C^{2n}$ given by $\vv \intprod (dZ; dW)$.  With this convention,
the Stenzel metric $g$ satisfies
\begin{equation}\label{gform}
g(\vv,\vw) = 2\Re(\overbar{\vv} \rG \vw^T),
\end{equation}
and its K\"ahler form $\Kahler = \ri \di \dibar f$ satisfies
\begin{equation}\label{omform}
\Kahler(\vv, \vw) = -2\Im (\overbar{\vv} \rG \vw^T),\quad \forall \vv, \vw \in T_\opoint \M,
\end{equation}

The process of verifying that $\Kahler$ vanishes on the tangent space to $\Phi(NM)$ is made simpler by
decomposing tangent vectors into vertical and horizontal pieces.  (By vertical vectors, we mean those
tangent to the images under $\Phi$ of the fibers of $T\CP^n$, and the horizontal vectors are those in the orthogonal
complement with respect to $g$.)  Computing
$$\left.\dfrac{d}{d s}\right\vert_{s=0} \Ahat \circ \Phihat(\zeta, \xi + s \eta)$$
when $\zeta = E_0$, $\xi = \ri t E_n$ and $\eta \in \C^{n+1}$ ranges over all vectors satisfying $\eta \cdot \overbar{\zeta}=0$, shows that
the space of vertical tangent vectors at $\opoint$ consists of all vectors of the form $(\vz; -\overbar{\vz})$ for $\vz \in \C^n$.
Then, noting the special form of $\rG$ in \eqref{Gpoint}, it is easy to check using \eqref{gform} that the space of horizontal vectors consists of all vectors of the
form $(\vz; \overbar{\vz})$.  It is also easy to check that $\Kahler(\vv, \vw)=0$ whenever
$\vv$ and $\vw$ are both vertical or both horizontal.

Equation \eqref{daffy} shows that the tangent space to $\Phi(NM)$ at $\opoint$ is spanned
by purely vertical vectors $(\ri \et\mu; \ri \trimb{e_\mu})$ and $(\Etn; -\Etn)$,
and the `mixed' vectors
$$\vu_\a := (\et\a - \tau^2 r_\a \Etn; \trimb{e_\a} - \tau^2 r_\a \Etn)
 -\tau h_{\a\b} (\ri \et\b; \ri \trimb{e_\b}).$$
(Note that the first term is horizontal and the second is vertical.)
When evaluating $\Kahler$ on the pairing of $\vu_\alpha$ with a purely vertical
vector, we only need the horizontal part of $\vu_\alpha$.  For example, we compute
using \eqref{Gpoint} and \eqref{omform} that
\begin{multline*}
\Kahler( \vu_\alpha,\,  (\Etn; -\Etn) )= \dfrac{-2}{1-\tau^4}
\Im\left[ (1+\tau^2)\left( (\trimb{e_\a} -\tau^2 r_\a \Etn) \cdot \Etn + (\et\a -\tau^2 r_\a \Etn)\cdot (-\Etn)\right)\right.
\\
+(q-\tau^2)\left( (\trimb{e_\a} -\tau^2 r_\a \Etn)\cdot \rM \Etn + (\et\a -\tau^2 r_\a \Etn)\cdot \rM (-\Etn)\right)
\\
-q \left.\left( (\trimb{e_\a} -\tau^2 r_\a \Etn)\cdot \rM (-\Etn) + (\et\a -\tau^2 r_\a \Etn)\cdot \rM \Etn\right)\right]=0.
\end{multline*}
In fact, the terms on each line inside the square brackets cancel out because $\rM \Etn = \Etn$ and because
$$0 = \langle e_\a, e_{2n}\rangle = \Re( \et\a \cdot \ri \Etn),$$
so that $\et\a \cdot \Etn$ is real (and equal to $-r_\a$).
Pairing with the other vertical vectors, we get
\begin{multline*}
\Kahler( \vu_\alpha,\,  (\ri \et\mu; \ri \trimb{e_\mu}) )=\dfrac{-2}{1-\tau^4}
\Im\left[ (1+\tau^2)\left( (\trimb{e_\a} -\tau^2 r_\a \Etn) \cdot \ri \et\mu + (\et\a -\tau^2 r_\a \Etn) \cdot \ri\trimb{e_\mu} \right)\right.
\\
+(q-\tau^2) \left( (\trimb{e_\a} -\tau^2 r_\a \Etn) \cdot \ri \rM \et\mu + (\et\a -\tau^2 r_\a \Etn) \cdot \ri\rM \trimb{e_\mu} \right)
\\
-q\left.\left( (\trimb{e_\a} -\tau^2 r_\a \Etn) \cdot \ri \trimb{e_\mu} + (\et\a -\tau^2 r_\a \Etn) \cdot \ri \et\mu \right)\right]
\end{multline*}
Again, $\et\mu \cdot \Etn=-r_\mu$ is real, and so $\rM \et\mu = -r_\mu \Etn$; we also note that
$$0 = \langle e_\a, e_\mu \rangle = \Re( \et\a \cdot \trimb{e_\mu}) = \tfrac12(\trimb{e_\a} \cdot \et\mu +
\et\a \cdot \trimb{e_\mu}).$$
Thus,
$$\Kahler( \vu_\alpha,\,  (\ri \et\mu; \ri \trimb{e_\mu}) ) = \dfrac{-4}{1-\tau^4}
\left[ (1+\tau^2)\tau^2 r_\a r_\mu + (q-\tau^2)(1+\tau^2) r_\a r_\mu - q(1+\tau^2) r_\a r_\mu\right] = 0.$$

It remains to check that $\Kahler$ is zero on a pair $\vu_\a, \vu_\b$ of `mixed' vectors.
Let
$$\vv_\a = (\et\a-\tau^2 r_\a \Etn; \trimb{e_\a} - \tau^2 r_\a\Etn),$$
the horizontal part of $\vu_\a$, and let $\vw_a = (\ri \et\b; \ri \trimb{e_\b})$.
  First, we compute that
\begin{multline*}
\Kahler( \vv_\a, \vw_\b)
=\dfrac{-2}{1-\tau^4}
\Im\left[ (1+\tau^2)\left( (\trimb{e_\a} -\tau^2 r_\a \Etn)\cdot \ri \et\b + (\et\a -\tau^2 r_\a \Etn) \cdot \ri \trimb{e_\b}\right) \right.
\\
+ (q-\tau^2) \left( (\trimb{e_\a} -\tau^2 r_\a \Etn) \cdot \ri \rM \et\b + (\et\a -\tau^2 r_\a \Etn) \cdot \ri \rM \trimb{e_\b}\right)
\\
- q \left.\left( (\trimb{e_\a} -\tau^2 r_\a \Etn) \cdot \ri \rM \trimb{e_\b} + (\et\a -\tau^2 r_\a \Etn) \cdot \ri \rM \et\b\right)\right]
\\
= \dfrac{-4}{1-\tau^4}\left[ (1+\tau^2)(\delta_{\a\b} + \tau^2 r_\a r_\b) + (q-\tau^2)(1+\tau^2) r_\a r_\b-q(1+\tau^2)  r_\a r_\b\right]
=\dfrac{-4}{1-\tau^2}\delta_{\a\b},
\end{multline*}
where we have used the fact that $\rM e_\b = -r_\b \Etn$ and
$\delta_{\a\b} = \langle e_\a, e_\b\rangle = \Re( \trimb{e_\a} \cdot \et\b).$
Then
\begin{align*}
\Kahler(\vu_\a, \vu_\b) &= \Kahler(\vv_\a - \tau h_{\a\gamma} \vw_\gamma, \vv_\b - \tau h_{\b\epsilon} \vw_\epsilon) \\
&= -\tau h_{\a\gamma} \Kahler(\vw_\gamma, \vv_\b)-\tau h_{\b\epsilon} \Kahler(\vv_\a, \vw_\epsilon) \\
&= \dfrac{-4}{1-\tau^2}\left(\tau h_{\a\gamma} \delta_{\gamma\b}-\tau h_{\b\epsilon} \delta_{\a\epsilon} \right)=0,
\end{align*}
where we use index ranges $1\le \alpha, \beta, \gamma, \epsilon \le k$.
\end{proof}

\section{The Austerity Condition}\label{ausect}
Let $\rS$ be the matrix on the right in \eqref{daffy}.  Then
the pullback of the holomorphic volume form $\Theta$ under $\A \circ \Phi\circ \imath$ equals  $\det \rS$
times a real volume form on $NM$.  In this section we will evaluate this determinant, and find conditions
under which it has constant phase.   In what follows, we again represent an arbitrary
normal vector $\bnu \in NM$ by its horizontal lift in $T_\vz \Mhat$, which is given by
$t e_{2n}$ for some $t\in \R$ and some adapted frame at $\vz$.

First, we consider the special case when $\bnu \in \calN$.  Then $r_\a=0$ for all $\a$, and we can factor $\rS$ as
\begin{equation}\label{holoS}
\begin{pmatrix}
I_k - \ri\tau H & 0 & 0 \\
0 & \ri\tau I_{2n-k-1} & 0 \\
0 & 0 & \ri (1-\tau^2)
\end{pmatrix}
\left(\begin{array}{ccc}
\et\a&;&\overbar{\et\a} \\
\et\mu&;&\overbar{\et\mu}\\
\ri \Etn&;&-\ri\Etn
\end{array}\right),
\end{equation}
where $I_k$ is the $k\times k$ identity matrix and $H$ is the matrix with entries $h_{\a\b}$.
Letting $V$ be the matrix on the right in \eqref{holoS}, we note that
\begin{equation}\label{Varrange}
\tfrac12 V \begin{pmatrix} I_n & -\ri I_n \\ I_n & \ri I_n \end{pmatrix}=
\begin{pmatrix} \split{\et\a}\\ \split{\et\mu}\\ \split{{\ri} \Etn}\end{pmatrix}.
\end{equation}
The matrix on the right-hand side of \eqref{Varrange} lies in $O(2n)$, but it has determinant $(-1)^n$, since
when we substitute our particular frame into \eqref{sonny}, we obtain
$$1= \det \begin{pmatrix} 1 & 0 \ldots &  0 & 0 \ldots  \\
                        0 & 0 \ldots & 1 & 0 \ldots \\
                        0 & \Re \et\a & 0 &\Im\et\a \\
                        0 & \Re \et\mu & 0 &\Im \et\mu \\
                        0 & 0 \ldots & 0 & \Etn \end{pmatrix}
= (-1)^n \det \begin{pmatrix}\Re \et\a &\Im\et\a \\
                           \Re \et\mu &\Im \et\mu \\
                           0 \ldots & \Etn \end{pmatrix}.
$$
Taking determinants on each side in \eqref{Varrange} gives $(\ri/2)^n \det V  = (-1)^n$, so that $\det V = (2\ri)^n$.
Thus,
\begin{equation}\label{detso}
\det \rS = (-2)^n \ri^{n-k}\tau^{2n-k-1}(1-\tau^2) \det(I_k - \ri \tau H).
\end{equation}

It is clear that the real part of $\ri^{n-k} \det\rS$ is nonzero for values of $\tau$ in an open interval containing zero.
On the other hand, by diagonalizing $H$ it is easy to see that
\begin{equation}\label{imdet}
\Im \det(I_k -\ri \tau H) = \sum_{j=1}^{\lfloor (k+1)/2\rfloor} (-1)^j \tau^{2j-1} H^{(2j-1)},
\end{equation}
where $H^{(2j-1)}$ denotes the elementary symmetric polynomial of degree $2j-1$ in the eigenvalues of $H$.
Thus, we conclude that
\begin{quote}
For $\bnu \in \calN$, the imaginary part of $\ri^{n-k} \Kahler$ vanishes along the
line in $NM$ spanned by $\bnu$ if and only if all odd-degree elementary symmetric
polynomials in the eigenvalues of $H$ vanish (where $H$ represents $\bnu \cdot \II$ with respect to an orthonormal basis).
\end{quote}

\begin{cor}\label{holcor}  If $M\subset \CP^n$ is a complex submanifold, then $M$ is austere.
\end{cor}
\begin{proof}  Because the complex structure $\rJ$ is parallel along $M$, the second fundamental form satisfies $\II(X,\rJ Y)=\rJ \II(X,Y) = \II(\rJ X, Y)$.  If matrix $J$ represents
the complex structure with respect to an orthonormal basis, then $HJ = J^T H= -JH$,
and hence $JHJ = -J^2 H=H$.  So,
$$\det(I_k - \ri \tau H) = \det(I_k - \ri\tau J^{-1}H J) = \det(I_k + \ri \tau JHJ) = \det(I_k + \ri \tau H),$$
and so $\Im\det(I_k - \ri\tau H)=0$.
\end{proof}

Now consider the more general case, when $\bnu$ has a non-zero component in $\calE$, and thus $\rJ \bnu$ has a
nonzero orthogonal projection onto $\calD$.  We will further specialize the orthonormal frame so that
$$\ri e_{2n} = \cos \theta\, e_1 + \sin\theta\, e_{2n-1},$$
where $\theta$ is the angle between $\rJ \bnu$ and the tangent space to $M$.
Since we also have $e_{2n}=\ri E_n$, then $r_1 =\cos\theta$, $r_{2n-1}=\sin\theta$ and all other $r_\alpha, r_\mu$ are zero.

To simplify calculating $\det\rS$, we modify the matrix by adding $\tau^2 r_1/(\tau^2-1)$ times row $2n$ to row $1$, giving
$$\rS' = \left(\begin{array}{ccc}
\et\a -\ri \tau \Ht\a &;& \trimb{e_\a} - \ri \tau \trimb{h_\a}-2\tau^2 r_\a \Etn \\
\ri \tau \et\mu &;& \ri\tau \trimb{e_\mu} \\
(\tau^2-1)\Etn &;& (1-\tau^2)\Etn
\end{array}
\right),
$$
where we introduce the abbreviation $\Ht\a = h_{\a\b} \et\b$.  Again, only coefficient $r_1=\cos\theta$ is nonzero, and we expand the determinant
in terms of it.  Letting $\rS_0$ denote the matrix in \eqref{holoS}, we have
$$\det\rS = \det\rS' = \det\rS_0 + (-1)^{2n-1}(-2\tau^2\cos\theta)
\det \left(\begin{array}{ccc}
\et\b -\ri \tau \Ht\b &;& \clipb{e_\b} - \ri \tau \clipb{h_\b} \\
\ri \tau \et\mu &;& \ri \tau \clipb{e_\mu} \\
(\tau^2-1) \Etn &;& 0
\end{array}\right),
$$
where we now use the index range $2\le \beta \le k$, and $\clip{\ }$ indicates the result of
deleting the first and last entries from a vector in $\C^{n+1}$.  Thus, the matrix on the right
is $2n-1 \times 2n-1$; moreover, in the bottom row only the $n$th entry is nonzero, so we may use a cofactor expansion to write
\begin{align*}\det\rS &=\det \rS_0 + 2\tau^2 \cos\theta (-1)^{2n-2+n-1} (\tau^2-1)
\det\left(\begin{array}{ccc}
\clip{e_\b} -\ri \tau \clip{H_\b} &;& \clipb{e_\b} - \ri \tau \clipb{h_\b} \\
\ri \tau \clip{e_\mu} &;& \ri \tau \clipb{e_\mu}\end{array}\right) \\
&=
\det\rS_0 + 2\cos\theta  (-1)^{n-3}\tau^2 (\tau^2-1) \det \begin{pmatrix} I_{k-1} -\ri\tau \clip{H} & 0 \\ 0 & \ri \tau I_{2n-k-1}\end{pmatrix}
\det\left( \begin{array}{ccc} \clip{e_\b} &;& \clipb{e_\b} \\ \clip{e_\mu} &;& \clipb{e_\mu}\end{array}\right),
\end{align*}
where $\clip{H}$ is the $(k-1)\times (k-1)$ matrix obtained from $H$ by deleting the first row and column.

\begin{lemma}\label{debt}Let $\clip{V} = \left( \begin{array}{ccc} \clip{e_\b} &;& \clipb{e_\b} \\ \clip{e_\mu} &;& \clipb{e_\mu}\end{array}\right)$.
Then $\det \clip{V} =(-2\ri)^{n-1}\cos\theta$.
\end{lemma}

Using this lemma (to be proved later), and the formula \eqref{detso} for $\det\rS_0$, we have
\begin{align*}
\det\rS &= \det\rS_0 + 2(-1)^{n-3}\cos^2 \theta \tau^2 (\tau^2-1)(\ri \tau)^{2n-k-1} (-2\ri)^{n-1}\det(I_{k-1} -\ri \tau \clip{H})\\
&= (-2)^n \tau^{2n-k-1} \left[ \ri^{n-k} (1-\tau^2) \det(I_k - \ri \tau H) +\ri^{2n-k-1}(-\ri)^{n-1}  \tau^2(1-\tau^2)  \cos^2\theta
\det(I_{k-1} -\ri \tau \clip{H})\right]\\
&=2^n \ri^{n-k} \tau^{2n-k-1}(1-\tau^2)\left[ \det(I_k - \ri \tau H)  +\tau^2 \cos^2\theta \det(I_{k-1} -\ri\tau \clip{H})\right].
\end{align*}
As in \eqref{imdet},
\begin{multline*}\Im \left[\det(I_k -\ri \tau H)+\tau^2 \cos^2\theta\det(I_{k-1} -\ri\tau \clip{H})\right]
\\
=  \sum_{j=0}^{\lfloor (k-1)/2\rfloor} (-1)^{j+1} \tau^{2j+1} H^{(2j+1)}
 +\cos^2\theta \sum_{j=1}^{\lfloor k/2 \rfloor}(-1)^j\tau^{2j+1} \clip{H}^{(2j-1)}\\
= -\tau H^{(1)} + \tau^3 (H^{(3)}-\cos^2 \theta \clip{H}^{(1)}) - \tau^5 (H^{(5)}-\cos^2 \theta \clip{H}^{(3)}) + \ldots
\end{multline*}
Thus, for $\bnu \notin \calN$ the imaginary part of $\ri^{n-k} \Hvol$ vanishes along the
line in $NM$ spanned by $\bnu$ if and only if $H^{(1)}=0$, $H^{(3)}=\cos^2 \theta \clip{H}^{(1)}$, and so on,
up to $H^{(k)} = \cos^2 \theta \clip{H}^{(k-2)}$ if $k$ is odd, or $0 = \cos^2\theta \clip{H}^{(k-1)}$ if $k$ is even.
On the other hand, if $\bnu \in \calN$ then $\cos\theta=0$ and these conditions simplify to the requirement that all the odd degree symmetric
polynomials in the eigenvalues of $H$ vanish---i.e., the conclusion we reached for the special case above.   We can therefore summarize our calculations as follows:

\begin{thm}\label{austerethm}
A submanifold $M\subset \CP^n$ of real dimension $k$ is austere if and only if, for every unit normal vector $\nu \in N_p M$ and every point $p \in M$,
\begin{equation}\label{Hcond}
H^{(2j+1)}=\cos^2\theta\, \clip{H}^{(2j-1)},\qquad  0\le j \le \lfloor k/2\rfloor,
\end{equation}
where $\theta$ is the angle between
$\rJ \bnu$ and $T_p M$, $H$ is the quadratic form $\nu \cdot \II$, $\clip{H}$ denotes the restriction of $\bnu \cdot \II$ to the subspace of $T_p M$ orthogonal to $\rJ \bnu$, and the odd-degree symmetric polynomials in \eqref{Hcond}
are understood to be zero if the degree is negative or larger than the size of the representative matrix.
\end{thm}

By setting $j=0$ in \eqref{Hcond}, we obtain:
\begin{cor} If $M \subset \CP^n$ is austere, then $M$ is minimal.
\end{cor}

The following examples may help us understand the austerity condition:

\begin{egg} Assume $M \subset \CP^n$ is a hypersurface.  Then $\calN$ has rank zero, $\calD$ and
$\calE$ have rank one, and $\rJ:\calD \to \calE$.  Hence $\theta=0$, and we may write the austere conditions as
$$A^{(2j+1)} = \clip{A}^{(2j-1)}, \qquad 0 \le j \le n-1,$$
where $A$ denotes the scalar-valued second fundamental form of $M$ and $\clip{A}$ is its restriction
to the holomorphic distribution $\calH$.
\end{egg}

\begin{egg}  Assume $M \subset \CP^n$ is a curve.  Then $M$ is austere
 if and only if it is a geodesic.
\end{egg}


\begin{egg}\label{surfex} Assume $M \subset \CP^n$ is a surface which is not a holomorphic curve.  Hence, $\calH$ has rank zero.
Then $M$ is austere if and only if $M$ is minimal and the ``highest-degree condition'', obtained by setting $j=1$ in \eqref{Hcond},
holds; this condition is that $\II(\vv, \vv)=0$, where $\vv \in T_p M$ is the vector orthogonal to $\rJ\bnu$, and
$\bnu$ runs over the unit circle bundle in $\calE_p$ for all $p \in M$.
\end{egg}
Austere surfaces are discussed in more detail in the next section.

\begin{proof}[Proof of Lemma \ref{debt}]
As observed after equation \eqref{Varrange},
$$(-1)^n = \det \begin{pmatrix} \split{\et\a}\\ \split{\et\mu}\\ \split{{\ri} \Etn}\end{pmatrix}
=\det \begin{pmatrix}
\Re \clip{e_1}  & -\cos \theta & \Im \clip{e_1} & 0 \\
\Re \clip{e_\b} & 0            & \Im \clip{e_\b} & 0 \\
\Re \clip{e_\lambda} &0        & \Im \clip{e_\lambda}& 0 \\
\Re \clip{e_{2n-1}}&-\sin\theta &\Im \clip{e_{2n-1}} & 0 \\
0           & 0                 & 0             & 1 \end{pmatrix},
$$
where we take index ranges $1 < \b \le k$ and $k < \lambda < 2n-1$.
Because the first and last entries of $e_\b$  and $e_\lambda$  are zero,
vectors $\clip{e_\b}$ and $\clip{e_\lambda}$ are mutually orthogonal unit vectors in $\C^{n-1}$.  Since $\clip{e_1}$ and $\clip{e_{2n-1}}$
are orthogonal to all of them, these two vectors must be linearly dependent over $\R$.  Since $\trim{e_{2n-1}}$ is a unit vector
with $\trim{e_{2n-1}} \cdot \Etn = -\sin\theta$, then $|\clip{e_{2n-1}}|^2 = \cos^2 \theta \ne 0$.  If we set $\clip{e_1} = a \clip{e_{2n-1}}$
for a scalar $a$, then solving $0 = \langle \trim{e_1}, \trim{e_{2n-1}} \rangle$ gives $a=-\tan\theta$.
Thus, adding $\tan\theta$ times the second-last row to the first row in the matrix gives
$$(-1)^n = \det \begin{pmatrix}
0  & -\sec \theta & 0 & 0 \\
\Re \clip{e_\b} & 0            & \Im \clip{e_\b} & 0 \\
\Re \clip{e_\lambda} &0        & \Im \clip{e_\lambda}& 0 \\
\Re \clip{e_{2n-1}}&-\sin\theta &\Im \clip{e_{2n-1}} & 0 \\
0           & 0                 & 0             & 1 \end{pmatrix}
    = (-1)^n \sec \theta \det \begin{pmatrix} \split{\et\b}\\ \split{\et\lambda}\\ \split{\et{2n-1}}\end{pmatrix}.
$$
On the other hand, as in \eqref{Varrange} we have
$$\tfrac12 \clip{V} \begin{pmatrix} I_{n-1} & -\ri I_{n-1} \\ I_{n-1} & \ri I_{n-1}\end{pmatrix}
= \begin{pmatrix} \split{\et\b}\\ \split{\et\lambda}\\ \split{\et{2n-1}}\end{pmatrix}.$$
Taking determinants on each side and solving gives the desired formula for $\det\clip{V}$.
\end{proof}

\section{Classification of Austere Surfaces}\label{surfsect}

\noindent
In this section we classify surfaces in $\CP^n$ that satisfy
the austere condition of Theorem \ref{austerethm}.

\begin{prop}\label{framep}  Let $M \subset \CP^n$ be an austere surface such that $\calH=0$ at every point.
Then $M$ is totally geodesic.
\end{prop}

\begin{proof} By assumption, the splitting \eqref{Tsplit} implies $TM = \calD$, $NM = \calE \oplus \calN$,
where $\calD, \calE$ are rank 2 and $\calD \oplus \calE$ is $\rJ$-invariant.  Let $p \in M$ be an arbitrary point
and let $\nu \in N_p M$ be an arbitrary unit normal vector.
At $p$ we will construct a orthonormal basis for $T_p \CP^n$ of the form
$(\ve_1, \rJ \ve_1, \ve_2, \rJ\ve_2, \ldots, \ve_n, \rJ \ve_n)$, which we will refer to as a {\em unitary frame}.

Fix a unit vector $\ve_1 \in \calD_p$, and let $\theta$ be the angle between $\rJ \ve_1$ and $\calD_p$.
This is the {\em K\"ahler angle}, which is independent of the choice of $\ve_1$  and nonzero by assumption.
Thus, $\rJ \ve_1$ has a nonzero orthogonal projection onto $\calE_p$; let $\vw$ be the unit vector
in the direction opposite to this projection, and let $\ve_2 \in \calE_p$ be a choice of unit vector
orthogonal to $\vw$.  Then $\ve_1, \rJ\ve_1, \ve_2$ are linearly independent, and thus
$(\ve_1, \rJ \ve_1, \ve_2, \rJ \ve_2)$ is an orthonormal basis for $\calD_p \oplus \calE_p$.
Since $\langle \vw, \rJ\ve_1 \rangle = -\sin\theta$, we must have $\vw = -\sin \theta\, \rJ\ve_1 \pm \cos \theta \rJ\ve_2$.
We adjust the sense of $\ve_2$ so that
\begin{equation}\label{wexpr}
\vw = -\sin \theta\, \rJ\ve_1 + \cos \theta \rJ\ve_2.
\end{equation}
If we define
$$\vv := \cos\theta \rJ\ve_1 + \sin \theta \rJ \ve_2$$
then $\vv$ is orthogonal to $\vw$ and $\ve_2$, and thus $(\ve_1, \vv)$ is an orthonormal basis for $\calD_p = T_p M$.

We now choose the remaining vectors of the unitary frame so that $\ve_3$ is the unit vector in the
direction of the orthogonal projection of $\nu$ onto $\calN_p$, and $\calN_p$ is spanned by $\ve_3, \rJ \ve_3, \ldots, \ve_n, \rJ\ve_n$.
Let $\psi$ be the angle between $\nu$ and $\calN_p$, and let $\II^\nu$ be the quadratic
form on $T_p M$ given by $\nu \cdot \II$.
If $\psi=0$ (i.e., $\nu \in \calN_p$) then the austere condition reduces to $\operatorname{tr} \II^\nu = 0$.
We are interested in what additional conditions arise when $\psi \ne 0$, so we assume this from now on.

Let $\vu$ be the unit vector in the direction of the orthogonal projection of $\nu$ onto $\calE_p$.  Then
\begin{equation}\label{defpsi}
\nu = \cos \psi \, \ve_3 + \sin\psi \, \vu.
\end{equation}
Let $\varphi$ be the angle such that $\vu = \cos \varphi \, \ve_2 + \sin \varphi \vw$.  When we substitute this,
and then \eqref{wexpr}, into \eqref{defpsi} we obtain
\begin{align}
\nu &= \cos\psi \, \ve_3 + \sin \psi (\cos \varphi \,\ve_2 + \sin \varphi\, \vw)  \label{longnu1}\\
&= \cos \psi \, \ve_3 + \sin \psi (\cos \varphi\, \ve_2 - \sin \varphi \sin \theta \, \rJ\ve_1 + \sin\varphi \cos \theta \, \rJ\ve_2).
\label{longnu2}
\end{align}
It is important to note that, while fixing $\ve_3$, we can vary the normal vector $\nu$ by varying the angles
$\psi$ and $\varphi$ independently.

As in Example \ref{surfex}, the austere condition requires that $M$ be minimal and that $\II^\nu(\xi,\xi)=0$, where $\xi \in T_p M$ is a unit vector orthogonal to $\rJ \nu$.
By computing $\rJ\nu$ using \eqref{longnu2}, one can verify that
\begin{equation}\label{xiexpr}
\xi = \cos\varphi \,\ve_1 - \sin \varphi\, \vv.
\end{equation}
To compute $\II^\nu$, we express the second fundamental form in our basis normal directions as
$$\II^{\vw}=\begin{pmatrix} a_1&b_1\\ b_1&-a_1\end{pmatrix}, \quad
\II^{\ve_2}=\begin{pmatrix} a_2&b_2 \\ b_2&-a_2\end{pmatrix}, \quad
\II^{\ve_3}=\begin{pmatrix} a_3&b_3\\b_3&-a_3\end{pmatrix},
$$
where each quadratic form is represented by a matrix with respect to the orthonormal basis $(\ve_1, \vv)$.

Now using \eqref{longnu1} and \eqref{xiexpr}, the condition $\II^\nu(\xi,\xi)=0$ can be expanded as
$$
\begin{bmatrix} \cos\varphi & -\sin \varphi \end{bmatrix}
\left( \sin \psi \sin  \varphi \II^\vw + \sin \psi \cos \varphi \II^{\ve_2} + \cos\psi \II^{\ve_3}\right)
\begin{bmatrix} \cos\varphi \\ -\sin \varphi\end{bmatrix}.
$$
Since this equation must be satisfied for all angles $\varphi$ and $\psi$ (provided $\sin\psi \ne 0$), we can take particular values.
Setting $\varphi =0$ yields
$$a_2 \sin \psi + a_3 \cos\psi = 0,$$
which can only hold for all $\psi$ if $a_2 = a_3=0$.  Setting $\varphi = \pi/2$ yields
$$a_1 \sin \psi + a_3 \cos \psi = 0,$$
so that $a_1$ must also vanish.  Taking $a_1=a_2=a_3=0$ into account, the condition becomes
$$\sin \varphi \cos \varphi ( \sin \psi (b_1 \sin\varphi  + b_2 \cos\varphi) + b_3 \cos \psi)=0 \qquad \forall \varphi, \psi,$$
which implies that $b_1=b_2=b_3=0$.

Since $\nu$ is an arbitrary normal direction, $\ve_3$ can range over all of $\calN_p$, and $p$ is arbitrary, we conclude
that $M$ is totally geodesic.
\end{proof}

\begin{thm}\label{surfthm}  If $M\subset \C P^n$ is a connected austere surface, then $M$ is either a holomorphic curve or an open subset
of a real projective plane $\RP^2 \subset \CP^2$, where $\CP^2$ is embedded as a complex linear subspace in $\CP^n$.
\end{thm}

\begin{proof}  If $M$ is austere, then $M$ is minimal and hence real-analytic.  Thus, the points where $T_p M$ is $\rJ$-invariant
form an open and closed subset of $M$.  By Proposition \ref{framep}, $M$ is either a holomorphic curve or totally geodesic.
If $M$ is totally geodesic, then by a well-known result of Wolf \cite{Wo}, $M$ is either an open set of a complex line in $\CP^n$, or the real part of a complex projective plane.
\end{proof}

\end{document}